\documentclass[reqno,a4paper,12pt]{amsart}
\usepackage{amsmath,amssymb,amsthm,amscd,amsfonts,amsbsy,bm}
\usepackage[english]{babel}
\usepackage[mathscr]{eucal}

\numberwithin{equation}{section}

\renewcommand{\b}{\beta}
\newcommand{\g}{\gamma}
\newcommand{\G}{\Gamma}
\renewcommand{\d}{\delta}
\newcommand{\D}{\Delta}
\newcommand{\e}{\epsilon}
\newcommand{\ve}{\varepsilon}
\newcommand{\z}{\zeta}

\renewcommand{\k}{\kappa}

\renewcommand{\l}{\lambda}
\renewcommand{\L}{\Lambda}
\newcommand{\m}{\mu}
\newcommand{\n}{\nu}
\newcommand{\x}{\xi}

\renewcommand{\t}{\tau}

\newcommand{\C}{{\mathbb C}}
\newcommand{\R}{{\mathbb R}}

\newcommand{\bb}{{\mathbf b}}

\newcommand{\vb}{{\mathbf v}}

\newcommand{\Bb}{{\mathbf B}}

\newcommand{\Db}{{\mathbf D}}

\newcommand{\Rb}{{\mathbf R}}

\newcommand{\Vb}{{\mathbf V}}

\newcommand{\Yb}{{\mathbf Y}}
\newcommand{\Zb}{{\mathbf Z}}

\newcommand{\AF}{\mathfrak A}

\newcommand{\GF}{\mathbf G}
\newcommand{\HF}{\mathbf H}

\newcommand{\KF}{\mathbf K}

\newcommand{\LF}{\mathfrak L}

\newcommand{\PF}{\mathbf P}

\newcommand{\QF}{\mathbf Q}

\newcommand{\SF}{\mathfrak S}
\newcommand{\tF}{\mathfrak t}
\newcommand{\TF}{\mathfrak T}
\newcommand{\wF}{\mathfrak w}

\newcommand{\Ac}{{\mathcal A}}

\newcommand{\Fc}{{\mathcal F}}
\newcommand{\Gc}{{\mathcal G}}
\newcommand{\Hc}{{\mathcal H}}

\newcommand{\Lc}{{\mathcal L}}

\newcommand{\Tc}{{\mathcal T}}
\newcommand{\Xc}{{\mathcal X}}

\newcommand{\Zc}{{\mathcal Z}}

\newcommand{\diag}{{\rm diag}\,}
\newcommand{\curl}{{\rm curl}\,}

\newcommand{\sgn}{\hbox{{\rm sign}}\,}

\newcommand{\Bn}{\mathbf{B^\circ}}
\newcommand{\Q}{\QF}
\newcommand{\Qa}{\overline{\QF}{}}
\newcommand{\QFc}{\overline{\QF^\circ}{}}
\newcommand{\pd}{\partial} 
\newcommand{\pa}{\bar{\partial}}

\newcommand{\Rm}{R_-}
\newcommand{\Rp}{R_+}

\newcommand{\bl}{\bm{[}}
\newcommand{\br}{\bm{]}}

\newcommand{\meas}{\operatorname{meas\,}}

\newtheorem{theorem}{Theorem}[section]
\newtheorem{proposition}[theorem]{Proposition}
\newtheorem{lemma}[theorem]{Lemma}

\theoremstyle{definition}

\theoremstyle{remark}

\begin{document}

\title[Moderately Decaying Perturbations]{On the Spectral
 Properties of the Landau Hamiltonian Perturbed by a moderately Decaying Magnetic Field}

\author[Rozenblum]{Grigori Rozenblum}
\address[G. Rozenblum]{Department of Mathematics \\
                        Chalmers University of Technology
                        Sweden}
                        \address[G. Rozenblum]{Department of Mathematics\\ University of Gothenburg
                        Sweden}
\email{grigori@chalmers.se}
\author[Tashchiyan]{Grigory  Tashchiyan}
\address[G. Tashchiyan]{Department of Mathematics
\\                    St. Petersburg University
for Telecommunications\\
St.Petersburg, 198504, Russia}
\email{grigori.tashchiyan@gmail.com}

\begin{abstract}
The Landau Hamiltonian, describing the behavior of a
quantum particle in dimension 2 in a constant magnetic
field, is perturbed by a  magnetic field with
power-like decay at infinity and a similar electric
potential. We describe how the spectral subspaces
change and how the Landau levels split under this
perturbation.
\end{abstract}

\date{\today}

\maketitle


\section{Introduction}
\label{intro}

 The paper is devoted to the study of the spectrum of
the Schr\"odinger and Pauli operators in the plane,
with nonzero constant magnetic field perturbed by a
smooth magnetic and electric fields that decay
power-like at infinity. It continues the paper
\cite{RozTa} by the authors, where the case of a
compactly supported perturbation was considered.

 The Landau Hamiltonian describing the
motion of the quantum particle in two dimensions under
the influence  of the constant magnetic field is one
of the classical models in quantum physics. The
spectrum (found first in \cite{Fock}, see also
\cite{LanLif}) consists of  eigenvalues with infinite
multiplicity lying at the points of an arithmetic
progression. These eigenvalues are traditionally
called Landau levels (LL) and the corresponding
spectral subspaces are called Landau subspaces.

A natural question arises, what happens with the
spectrum of the Landau Hamiltonian under the
perturbation by a weak electrostatic potential or/and
magnetic field. One should expect that the Landau
levels split, and the problem consists in describing
quantitatively this splitting as well as in studying
the behavior of the spectral subspaces under the
perturbation.

The case of the perturbation by an electric potential
$V$, moderately, power-like decaying at infinity,  was
first studied by Raikov in \cite{Raikov1}.
 The case of a
fast decaying (or compactly supported) electric
potential was dealt with much later, in \cite{RaiWar},
see also \cite{MelRoz}. It was found that the
character of the Landau levels splitting depends
essentially on the rate of decay of $V$ and the
asymptotics of the eigenvalues in clusters can be
expressed in the terms of $V$, quasi-classically or
not.

When the magnetic field is also perturbed, the
situation becomes more complicated since is not the
magnetic field itself  but its potential that enters
in the quantum Hamiltonian. So, the perturbation of
the operator turns out to be
 fairly strong  even for a compactly supported
perturbation of the field and it   may even be not
relatively compact if the perturbation goes to zero at
infinity  not sufficiently fast. Iwatsuka \cite{Iwats}
proved that the invariance of the essential spectrum
still takes place, so Landau levels are the only
possible limit points of eigenvalues lying in the gaps
between them. Further on, in \cite{FilPush}, \cite{IwaTam1},
\cite{IwaTam2},  \cite{Raikov2} the character of the splitting of the
lowest Landau level was investigated. For compactly
supported  magnetic field perturbation  and electric
potential the splitting of all Landau levels was
studied in \cite{RozTa}. The main result of
\cite{RozTa} was the description of the spectral
subspaces of the perturbed operator corresponding to
the clusters around Landau levels. It was found that
these subspaces change  fairly strongly, and a rather
exact approximation for these subspaces was found in
the terms of modified creation and annihilation
operators. At the same time, although the perturbation
of the operator may be very strong, the splitting of
eigenvalues is super-exponentially weak, just like it
is  in the case of a perturbation by a compactly
supported electric field only.

In  the present paper we continue the study of the
Landau Hamiltonian with a perturbed magnetic field.
Now we  consider  the case of the perturbation
decaying moderately, power-like, at infinity. For the
spectral subspaces the results similar to the ones in
\cite{RozTa} hold. As for the asymptotics of the
eigenvalues in clusters, we obtain more complete
results, proving estimates and, under some natural
conditions, the power-like asymptotic formulas for
eigenvalues. The problem of the splitting of Landau levels under moderately decaying perturbations of the magnetic field has been considered in \cite{Ivrii}, see Theorem 11.3.17 there, where even the remainder term of the asymptotics of the eigenvalues in clusters was found. Powerful methods of microlocal analysis are used in \cite{Ivrii}. Our methods are more elementary, give the approximation for the spectral subspaces,  moreover,   we do not need the rather restrictive 'hyperbolicity' condition (11.3.49) imposed in \cite{Ivrii}.

The paper is heavily based upon the
methods and results of the papers  \cite{IwaTam1} and
\cite{RozTa}. Following the structure of the latter
paper, we   refer rather shortly to the fragments that
should be repeated, word for word or with minor
changes only, in order to be able to concentrate on
important differences. For convenience of references,
the numbering of sections here coincides with that in
\cite{RozTa}.

The second author was partly supported by a grant from
the Royal Swedish Academy of Sciences, he also thanks
Chalmers University of Technology in Gothenburg for
hospitality.

\section{Magnetic Schr\"odinger  and Pauli operators}
\label{setting}

\subsection{The unperturbed operators} We will denote the points in the plane
$\R^2$ by $x=(x_1,x_2)$; it is convenient to identify
$\R^2$ with $\C$ by setting $z=x_1+i x_2$. So, the
Hilbert space $L_2(\R^2)$ with   Lebesgue measure
(which will be denoted by $dx$) is identified with
$L_2(\C)$. The derivatives are denoted by  by
$\pd_k=\pd_{x_k}$ and we set, as usual,
$\pa=(\pd_1+i\pd_2)/2,\; \pd=(\pd_1-i\pd_2)/2.$

The constant magnetic field is denoted by  $\Bn>0$.
The corresponding magnetic potential is
$A^{\!\circ}(x)=(A_1^\circ,A_2^\circ)=\frac{\Bn}{2}(-x_2,x_1)$.
Then the (unperturbed) magnetic Schr\"odinger operator
in $L_2(\R^2)$ is
\begin{equation}\label{2:Schr.0}
    \HF^\circ=-(\nabla+iA^{\!\circ})^2.
\end{equation}
  The Pauli operator describing the motion of a
spin-${\frac12}$ particle acts in the space of
two-component vector functions, it has the diagonal
form, $\PF^\circ=\diag(\PF^\circ_+,\PF^\circ_-)$,
where $\PF^\circ_\pm=\HF^\circ \pm \Bn$.

The spectrum of the Schr\"odinger operator is
described by the  classical construction originating
in \cite{Fock}.
  For the  complex magnetic potential
$\Ac^\circ=A_1^\circ+iA_2^\circ$ the \emph{creation
and annihilation operators} are introduced,
\begin{equation}\label{2:Creation.0}
   \QFc=-2i\pd-\overline{\Ac^\circ},\; \Q^\circ=-2i\pa-{\Ac^\circ}.
\end{equation}
These operators can be also expressed by means of the
scalar potential, the function
$
\Psi^\circ(z)=\frac{\Bn}4|z|^2,
 $ solving
the equation $\D\Psi^\circ=\Bn$:
\begin{equation}\label{2:Creation.01}
\nonumber \Q^\circ=-2i e^{-\Psi^\circ}\pa
e^{\Psi^\circ},\;
 \QFc=-2i e^{\Psi^\circ}\pd
e^{-\Psi^\circ}.
\end{equation}

  The operators   $\QFc ,\Q^\circ $ satisfy the following
basic relations
\begin{equation}\label{2:CommRel.0.1}
[\Q^\circ,\QFc]=2\Bn.
\end{equation}
\begin{equation}\label{2:CommRel.0.2}
\PF^\circ_+=\Q^\circ\QFc, \PF^\circ_-=\QFc\Q^\circ,
\HF^\circ=\Q^\circ\QFc-\Bn=\QFc\Q^\circ+\Bn.
\end{equation}
The spectrum of  $\HF^\circ$ is described in the
following way. The equation $\PF^\circ_-u=0$, $u\in
L_2$ is equivalent to ${\Q^\circ}u=e^{-\Psi^\circ}\pa(
e^{\Psi^\circ}u)=0.$

This means that  $f=e^{\Psi^\circ}u$ is an entire
analytical  function  in $\C$, such that after being
multiplied  by $e^{-\Psi^\circ}$ it belongs to $L_2$.
The space of such  functions $f$ is called Fock or
Segal-Bargmann space $\Fc$.

  So, the null
subspace of the operator $\PF^\circ_-$, i.e., its
spectral subspace corresponding to the eigenvalue
$\L_0=0$, is $\Lc_0=e^{-\Psi^\circ}\Fc$.  After this,
by  the commutation relations \eqref{2:CommRel.0.1},
\eqref{2:CommRel.0.2},  $\Lc_q=\QFc\Lc_0$ are the
spectral subspace of $\PF^\circ_-$ with eigenvalues
$\L_q=2q\Bn$, $q=0,1,\dots$, called \emph{Landau
levels}, and the spectra of $\HF^\circ$, $\PF^\circ_+$
consist, respectively, of $\L_q+\Bn$ and $\L_q+2\Bn$.
The operators $\QFc, \Q^\circ$ act between
\emph{Landau subspaces} $\Lc_q=\QFc^q \Lc_0,
q=0,1,\dots$,
 \begin{equation}\label{2:CrAnn}\QFc:
\Lc_q\mapsto \Lc_{q+1},\; \Q^\circ :\Lc_q\mapsto
\Lc_{q-1},\; \Q^\circ:\Lc_0\mapsto 0,\end{equation}
 and are, up to
constant factors, isometries of Landau subspaces.

The spectral projection $P_q^\circ$ of $\PF^\circ_-$
corresponding to the eigenvalue $\L_q=2q\Bn$ can be
thus expressed as
\begin{equation}\label{2:Projections}
    P_q^\circ=C_q^{-1}\QFc^q P_0^\circ {\Q^\circ}^q,
    \; C_q=q!(2\Bn)^q, \ q=0,1,\dots .
\end{equation}

\subsection{The perturbed operator}

We introduce the convenient  class of functions. A
function $F(x)$ is said to belong to the class $S_\b$,
 $\b<0,$ if \begin{equation}\label{2:MagFPert1}
    F(x)=O(|x|^{\b}), \ \partial_x^k
    F(x)=O(|x|^{\b-\d}), \ |x| \to\infty,\
    k=1,2,\dots,
\end{equation}
for some $\d\in(0,-\b).$ The particular value of $\d$
is irrelevant. So, if $F\in S_\b$ then $\partial_j F$
can be considered as a function in
$S_{\b-\frac{\d}{2}}$ (with $\frac{\d}{2}$ acting as
$\d$).

  Now we introduce the perturbation $\bb\in
C^\infty(\R^2)$ of the magnetic field and  set
$\Bb=\Bn+\bb$.  We suppose that
\begin{equation}\label{2:MagFPert2}
    \bb\in S_\b {\rm \  for \ some\  } \b<-2.
\end{equation}
 Under the condition \eqref{2:MagFPert2}, the upper
estimate for the counting function of the eigenvalues
in the clusters will be proved, as well as the
approximate representation of the spectral subspaces,
following \cite{RozTa}. For obtaining asymptotic
formulas, additional conditions will be imposed on
$\bb$ and eventually on the electric perturbation.

 Let $\psi$ be a scalar potential for the field
$\bb$, a solution of the equation $\D\psi=\bb$. Of
course, $\psi$ is defined up to a harmonic summand,
the choice of $\psi$ corresponds to the choice of
gauge. The magnetic potential
$a=(a_1,a_2)=(-\pd_2\psi,\pd_1\psi)$, $\curl a=\bb$,
is thus determined up to a gradient, and the complete
scalar and vector magnetic potentials are
\begin{equation}\label{2:potentials}
   \Psi=\Psi^{\circ}+\psi,\; A=A^{\circ}+a.
\end{equation}

We define the perturbed magnetic Schr\"odinger
operator as in \eqref{2:Schr.0}, with $A^\circ$
replaced by $A$:
  $  \HF=-(\nabla+iA)^2,$
and the components of Pauli operator as $
\PF_\pm=\HF\pm\Bb. $ It is easy to observe that the
difference between $\HF$ and $\HF^{\circ}$ contains an
operator of multiplication by $A^\circ\cdot a$, the
latter function does not decay at infinity, and thus
looks like being \emph{not} a relatively compact
perturbation of $\HF^{\circ}$. However, thanks to the
special form of this term, the perturbation is still
relatively compact if $a\to 0$ at infinity, as  was
noticed by Besch \cite{Besch}. In our case, under the
conditions imposed on $\bb$, the scalar potential
grows at most logarithmically, and the vector
potential $a$ decays as $|x|^{-1} $
 at infinity.

Now let us look  at the algebraic structure related
with the perturbed operators. The perturbed creation
and annihilation operators are defined similarly to
\eqref{2:Creation.0}: $ \Qa=-2i\pd-\overline{\Ac},\;
\Q=-2i\pa-\Ac,$ where $\Ac$ is the complex magnetic
potential, $\Ac=\Ac^\circ +(a_1+ia_2)$. The
commutation relations for $\Qa,\Q$ have the form
\begin{equation}\label{2:ComRel.1}
[\Q,\Qa]=2\Bb=2\Bn+2\bb,
\end{equation}
and $\PF_\pm$ and $\HF$ satisfy
\begin{gather}\label{2:ComRel.2}
\PF_+=\Q\Qa,\; \PF_-=\Qa\Q, \;\PF_+-\PF_-=2\Bb=2\Bn+2\bb,\\
\HF=\Q\Qa-\Bb=\Qa\Q+\Bb.\label{2:ComRel.3}
\end{gather}

The relations \eqref{2:ComRel.1}-\eqref{2:ComRel.3}
 contain variable functions on the
right-hand side and the spectra of Schr\"odinger and
 Pauli operators
 do not  determine each other any more. The only
information that one can obtain immediately, is the
description of the lowest point of the spectrum of
$\PF_-$.  Since, again, $\Qa=\Q^*$, the equation
$\PF_-u=0$ is equivalent to $\Q u=0$, or
$\pa(\exp(\Psi) u)=0$. So the function $f=u\exp\Psi$
is an entire analytical  function such that
$u=\exp(-\Psi) f\in L_2$. The space of entire
functions with this property is, obviously,
infinite-dimensional, it contains at least all
polynomials in ${z}$ variable, although it does not
necessarily coincide with the Fock space. We denote
the null-space of $\PF_-$, the space of \emph{zero
modes}, by $\Hc_0$. It is infinite-dimensional;
complex polynomials times $\exp(-\Psi)$ form a dense
set in $\Hc_0$. The lowest Landau level $\L_0$ is an
isolated point in the spectrum of $\PF_-$.

As it follows from the  relative compactness of the
perturbation, by Weyl's theorem, the essential
spectrum of the perturbed operator $\PF_-$ consists of
the same Landau levels $\L_q$, and the eigenvalues in
the gaps  may only have $\L_q$ as their limit points.
This latter fact was established much earlier by
Iwatsuka \cite{Iwats}, and the constructions in
\cite{RozTa} can  be considered as the extension of
the approach in \cite{Iwats}.

Now we add a perturbation by the electric potential.
Let $V(x)$ be a real valued function in $S_\b, \
\b<-2$. We introduce the operators
\begin{equation}\label{2:ElField}
    \HF(V)=\HF+V, \; \PF_\pm(V)=\PF_\pm+V.
\end{equation}
Since the operator of multiplication by $V$ is
relatively compact with respect to $\HF, \PF_\pm$, the
operators \eqref{2:ElField} have the same essential
spectra as the respective unperturbed ones ($\L_0$
ceases to be an isolated point of the spectrum of
$\PF_-(V)$ ). In this paper we are going to study the
distribution of the eigenvalues of $\HF(V),
\PF_\pm(V)$ near $\L_q$.
\subsection{Some resolvent and commutator estimates}
In the course of our proof we will need some
boundedness    property of the resolvent of the
operators $\PF_+, \PF_-, \HF$ and their spectral
projections. These properties, in a slightly less
general form, were established in \cite{IwaTam1},
Lemma 1.4.

We denote by $R_{\pm}(z)$ the resolvent of the
operators $\PF_{\pm}$ and by $\Pi_j$, $j=1,2,$ the
operators $\Pi_j=i\partial_j+(A_j+a_j)$.

\begin{proposition} \label{2:commutatorEstimate}
Suppose that $V,\bb$ belong to $S_\n,\ \n<-2$. Let $P$
be the spectral projection of the operator $\PF_{\pm}$
or $\HF$ corresponding to some bounded isolated piece
of the spectrum. Then for any real $l$ the following
operators are bounded: $\langle
x\rangle^{-\n+\d-l}[V,P]\langle x\rangle^l, \langle
x\rangle^{-\n+\d-l}[\bb,P]\langle x\rangle^l,$ $
\langle x\rangle^{-l} R_{\pm} \langle x\rangle^l $,
$\langle x\rangle^{-l} \Pi_{j}R_{\pm} \langle
x\rangle^l$.
\end{proposition}
The proof in \cite{IwaTam1} is given for the class
$S_\n$ with $\d=1$, for the projection $P$
corresponding to the lowest Landau level and for
positive $l$ only. In our formulation, the proposition
is proved in an analogous way.

 Next  we establish the
estimates for eigenvalues and singular numbers of some
compact operators. The estimates will be needed
further on, in the process of proving the required
eigenvalue asymptotics.

For a compact operator $T$ we, as usual, denote by
$n(\l,T)$ the distribution function of the singular
numbers ($s$-numbers) of $T$, i.e. the quantity of $s$- numbers of $T$ that are bigger than $\l$. If the operator is
self-adjoint, the distribution functions for the
positive and negative eigenvalues of $T$ are denoted
by $n_\pm(\l,T)$. The operator $T$ can be dropped from
the notation if this does not cause misunderstanding.

\begin{proposition}\label{4a:low order} Let $V$ be a function in $S_\n$, $\n<0$.
Consider the operator $X=X_N(V)= V\PF_+^{-N}$. Then
for $N$ sufficiently large,
\begin{equation}\label{4a:low order1}
    n(\l,X)=O(\l^{\frac{2}{\n}}),\ \l\to 0.
\end{equation}
Moreover, if $K$ is a compact operator then
\begin{equation}\label{4a:low order2}
n(\l,KX),\ n(\l,XK) =o(\l^{\frac{2}{\n}}),\ \l\to 0.
\end{equation}
\end{proposition}

Compared with  Proposition \ref{5:EstimateAboveL}
below, the above estimate shows that for $N$ large
enough, the operator $X_N(V)$ admits the same spectral
estimates as the Toeplitz type operator, with
$\PF_+^{-N}$ replaced by the spectral projection of
$\PF_\pm$.
\begin{proof}
Consider a very special case first. Let $V\in S_\n,
\n\in(-\frac12,0)$ and $N\in(-\n,\frac12]$. We will
obtain the $s$-numbers estimates of the operator $V
\PF_+^{-N}$. These numbers are majorated by the
s-numbers of the analogous  operator with $V$ replaced
by $V_\n= \langle x\rangle^{\n/2}$. The semigroup generated by $\PF_-{-N}$, by the diamagnetic inequality, is dominated by the semigroup generated by $(1-\Delta)^{\frac{N}2}$, therefore,
by the results of
\cite{RozDom} (see Theorems 1, 4 and Sect.5.4 there),
 these singular
numbers are majorated by  the $s$-numbers of the
operator $T_\n=V_\n (1-\Delta)^{-\frac{N}2}$. By the
Fourier transform, this operator is unitary equivalent
to $(1-\Delta)^{\n}\langle\xi\rangle^{-\frac{N}2}$,
and for the latter operator the required eigenvalue
estimate is given by Cwikel's theorem in
\cite{Cwikel}.

For the general case, again for $V=V_\n$, $\n<0$, we
take  $N$ so large that $-\n/N<\frac14$. We represent
$V$ in the form $V=\tilde{V}^{N+1}$,
$\tilde{V}=V^{\frac{1}{N+1}}$. Now in the expression
$\tilde{V}^{N+1}\PF_+^{-N}$ we leave one copy of
$\tilde{V}$ in the first (utmost left) position and
start moving the remaining copies to the right,
commuting them  with copies of $\PF_+^{-1}$ in such
way that finally there will be only one entry of
$\tilde{V}$ or its derivatives between two copies of
$\PF_+^{-1}$. As a result, we arrive to a  collection
of summands, each being the product of operators of
the form  $W\PF_+^{-\frac12}$, $\PF_+^{-\frac12}W$,
and, possibly, some more bounded operators of the form
considered in Proposition \ref{2:commutatorEstimate},
where $W\in S_{\n/N}$ is $\tilde{V}$ or some of its
derivatives. For the operators $W\PF_+^{-\frac12}$,
$\PF_+^{-\frac12}W$ we can apply the estimate found in
the first part of the proof and obtain the required
inequality using the Weyl inequality for the s-numbers
of the product of operators. The second statement
follows from the first one and, again, the Weyl
inequality.
\end{proof}

\section{Approximate spectral subspaces}
\label{Subspaces3} In this section, under the
condition that  $\bb$ satisfies \eqref{2:MagFPert2},
we construct the approximate spectral subspaces of the
operators \eqref{2:ElField}. This is done in the same
way as in \cite{RozTa}, so we just briefly describe
the construction and
 pinpoint  the
main differences.

First of all, we consider the   null subspace $\Hc_0$.
In \cite{RozTa} it is shown that for $\bb\in
C_0^{\infty}$, $\Hc_0$ possesses a dense subspace of
rapidly decaying functions. The same reasoning proves
this property  for our case. It is here, that the
condition $\b<-2$ implying that $V,\bb\in L_1$ is
essential.

Let $\d_q=(\L_q-\g,\L_q+\g), q=0,1,2,\dots$, $\g<\Bn$,
be intervals of the same size on the real axis,
centered at the Landau levels $\L_q=2q\Bn$. We choose
the size of $\d_q$ in such way that  neither of these
intervals has the eigenvalues of $\PF_-$ at its
endpoints. Moreover, since the lowest LL $\L_0=0$ is
an isolated point of the spectrum of $\PF_-$, we can
choose the size of the intervals in such way that
$\d_0$ contains only this point of spectrum.  We
denote by $\Hc_q$ the spectral subspace of $\PF_-$
corresponding to the interval $\d_q$ and by $P_q$ the
corresponding spectral projection. Since, by
\cite{Iwats}, the spectrum of $\PF_-$ is discrete
between Landau levels, the change of $\d_q$ leads only
to a finite-rank perturbation of $P_q$. As usual, the
spectral projection $P_q$ can be be expressed by means
of the integration of the resolvent of $\PF_-$ along a
closed contour $\G_q$ in the complex plane, not
passing through the eigenvalues of $\PF_-$ and
containing inside only those eigenvalues that lie in
$\d_q$. Again, using the discreteness of the spectrum
of $\PF_-$ between the Landau levels, we can choose
these contours so that  they are obtained from $\G_0$
by the shift along the real axis in the complex
plane, $\G_q=\G_0+2q\Bn$.

Now we are going to establish several properties of
the subspaces $\Hc_q$, projections $P_q$ and some
related operators. First, note the simple fact
following directly from the spectral theorem.

\begin{proposition}\label{2:PropUnpertSubspace}For any
$q=0,1,\dots,$ and any polynomial $p(\l)$ the operator
$p(\PF_-) P_q$ is bounded, moreover
$(p(\PF_-)-p(\L_q))P_q$ is compact.
\end{proposition}

In fact, by the spectral theorem the nonzero spectrum
of the operator $p(\PF_-)P_q$ consists of the points
$p(\l_j)$ where $\l_j$ are all points of spectrum of
$\PF_-$ in $\d_q$ and thus  all $p(\l_j)$  live in a
bounded interval. Moreover, $p(\l_j)$ may only have
$p(\L_q)$ as their limit point by Iwatsuka's theorem.

The following lemma  will enable us later to prove a
much stronger compactness property.
\begin{lemma}\label{2:LemmaBound} Let each of
 $T_j,\; j=1,\dots,N$ be one of operators $\Q$ or
$\Qa$. Then for some constants $C, C'$, for any $u$ in
the domain of the operator $\PF_-^N$,
\begin{equation}\label{2:LemmaBound.1}
   \|T_1T_2\dots T_N u\|^2\le C(\PF_-^Nu,u)+C'\|u\|^2.
\end{equation}
\end{lemma}
The proof of Lemma~\ref{2:LemmaBound} for the case
$\bb\in C_0^\infty$ can be found in \cite{RozTa}
Section 7, in our case the proof goes exactly in the
same way.

Now we can establish the compactness property.
\begin{proposition}\label{2:PropComp} Let
$T_1,\dots,T_N$ be a collection of operators, each
being $\Qa$ or $\Q$, and let $h_j$, $j=0,\dots N$ be
functions with all derivatives bounded, $\Tc=h_0T_1h_1
\dots T_N h_N$.
 Then for any
$q$ and for any polynomial $p(\l)$ the operators
$\Tc(p(\PF_-)-p(\L_q))P_q \Tc$
 are compact.\end{proposition}
\begin{proof}
By commuting functions $h_j$ and operators $T_j$
(moving all functions to the left), we transform
 the left operator $\Tc$ to the sum of
terms of the form $\tilde{h}_\k \Tc_\k$ where
$\tilde{h}_\k$ are bounded functions and $\Tc_\k$ is a
product of no more than $N$ operators $\Qa, \Q$.
Similarly, in $\Tc$ that stands to the right of $P_q$,
we move all functions to the utmost right positions,
to get the representation of $\Tc$ as a sum of terms
$\Tc_\varkappa\tilde{h}_\varkappa$, being the product
of a bounded smooth function $\tilde{h}_\varkappa$ and
no more than $N$ operators $\Qa, \Q$.

For each of the terms arising in this way in the
decomposition of $\Tc(p(\PF_-)-p(\L_q))P_q \Tc$, we
can write
\begin{gather}\label{2:ProComp.1}
\tilde{h}_\k\Tc_\k (p(\PF_-)-p(\L_q))P_q
\Tc_\varkappa\tilde{h}_\varkappa = \bl\tilde{h}_\k
\Tc_\k(\PF_-^N+1)^{-1}\br\times \\\notag\bl(\PF_-^N+1)
(p(\PF_-)-p(\L_q))P_q(\PF_-^N+1)\br
\bl(\PF_-^N+1)^{-1}\Tc_\varkappa\tilde{h}_\varkappa\br.
\end{gather}

In \eqref{2:ProComp.1}, the first factor in brackets
is bounded by Lemma~\ref{2:LemmaBound}, and the middle
factor is compact by
Proposition~\ref{2:PropUnpertSubspace}. The last
factor in brackets is also bounded,
 by Lemma~\ref{2:LemmaBound} applied to the adjoint operator.
  \end{proof}

Now we  describe the main construction of the paper,
the approximate spectral subspaces of the perturbed
operator. It is sufficient to consider the operator
$\PF_-(V)$. In fact, by \eqref{2:ComRel.1},
\eqref{2:ComRel.2}, \eqref{2:ComRel.3},
\begin{equation}\label{3:DifferentOperators}\HF(V)=\PF_-(V+\bb)+\Bn,
\PF_+(V)=\PF_-(V+2\bb)+2\Bn,\end{equation}
 and thus
these operators differ from $\PF_-(V)$ by a shift and
by the electric type perturbations $\bb$, $2\bb$. We
find the approximate spectral subspaces of $\PF_-$.
Adding an electric perturbation will then be an easier
task.

The subspaces approximating $\Hc_q$ will be defined
as:
\begin{equation}\label{3:Subspaces}
    \Gc_0 =\Hc_0,\; \Gc_q=\Qa^q\Gc_0,\; q=1,2,\dots.
\end{equation}
So we mimic  the construction of the eigenspaces of
the unperturbed Landau Hamiltonian, see
\eqref{2:CrAnn}, in the same way as it was done in
\cite{RozTa}, by applying the creation operators to
the space of zero modes.

Of course, since we apply the unbounded  operator
$\Qa$, we must show that we never leave the space
$L_2$,
 and moreover,  that the subspaces
$\Gc_q$ are closed. Both these properties, as well as
some other results  will be based upon the
 important Proposition~\ref{3:basicPropositionBP} (an analogy of Proposition
3.4 in \cite{RozTa}).

The essential difference is that now, when the
perturbation does not have compact support, we have to
trace the rate of decay of different terms arising in
the process of transformations. This analysis enables
us to single out the leading terms in the resulting
expansions.

\begin{proposition}\label{3:basicPropositionBP}
Let  $q>0$.\begin{enumerate}
 \item There exists a
function $\Zc_q[\bb]\in S_\b$ depending only on $q$,
$\Bn$, and $\bb$ such that for any $u\in\Hc_0$,
\begin{equation}\label{3:BP:equation}
    \|\Qa^q u\|^2=C_q \|u\|^2+(\Zc_q[\bb] u,u),\; C_q=q!(2\Bn)^q.
\end{equation}
The function $\Zc_q[\bb]$ is  a polynomial in $\bb$
and its derivatives up to the order $2q-2$ with
coefficients depending on $\Bn$.  The  term linear in
$\bb$ and not containing derivatives equals
$C'_q\Bn^{q-1}\bb$, $C'_q=2^q q! q$. Moreover,
\begin{equation}\label{3:BP:equationSlow}
\Zc_q[\bb]-C'_q\Bn^{q-1}\bb=O(|x|^{\b-\d})
\end{equation}
at infinity.
 \item Let $U(x)$ be a function in $S_\b, \b<-2$.
  There exists a function $\Xc_q[\bb,U]\in S_\b$ depending
 only on $q$, $\Bn$,  $\bb$, and $U$ such that for any
$u\in\Hc_0$.
\begin{equation}\label{3:BP:equation.V}
  (U\Qa^q u,\Qa^q u)= (\Xc_q[\bb,U] u,u).
\end{equation}
The function $\Xc_q[\bb,U]$ is expressible as an order
$2q$ linear differential operator acting on $U$, with
coefficients depending polynomially on $\bb,$ its
derivatives, and $\Bn$, moreover,
\begin{equation}\label{3:BP:equation.V.Slow}
\Xc_q[\bb,U]-C_q'\Bn^q U=O(|x|^{\b-\d})
\end{equation}
at infinity.
\end{enumerate}
\end{proposition}

\begin{proof} The combinatorial part of the proof, consisting of multiple commuting
 of the creation and annihilation operators with
 functions and with each other is exactly the same as
 in \cite{RozTa}. What remains  to be checked are the estimates
 \eqref{3:BP:equationSlow} and
 \eqref{3:BP:equation.V.Slow}. These estimates follow
 from the fact that all terms in the expressions $\Zc_q[\bb],\
 \Xc_q[\bb,U]$,
except the leading ones,
 contain either derivatives of $\bb$ or $U$, or products of
 these functions, and therefore decay at infinity
 not slower than $|x|^{\b-\d}$.
\end{proof}

When applying  Proposition~\ref{3:basicPropositionBP}
and similar results, we need  a certain compactness
property. Such facts were used persistently in
\cite{RaiWar}, \cite{MelRoz}, but for the case of a
constant magnetic field only.
\begin{lemma}\label{3:LemComp}
Let $W(x)$ be a function in $S_\n, \ \n<0$. Let $\LF$
be an arbitrary  differential operator having the form
\begin{equation}\label{3:LF.Slow}
    \LF=f_1T_1f_2T_2\dots T_m,
\end{equation}
where each of $T_j$ is one of operators $\Q, \Qa$ and
$f_j$ are functions, with all derivatives bounded.
 Then the quadratic form
 \begin{equation}\label{3:Qform}\wF[u]=\int W(x)| \LF u|^2dx \end{equation}
 is compact in the space $\Hc_0$.\end{lemma}

 \begin{proof}
Let $u=e^{-\Psi}h$ be a function in
 $\Hc_0$, so $h(z)$ is an analytical function.
We write the quadratic form \eqref{3:Qform} as

\begin{equation}\label{3:Qform1}
\wF[u]=(W \LF u, \LF u).
\end{equation}
Now we move all the operators $\Qa$ in $\LF$ from the
second factor in \eqref{3:Qform1} to the first one and
from the first factor to the second one; thus they
turn into $\Q$. In the process of commuting these
$\Qa$ with $\Q$ and the functions $W$ and $f_j$, some
derivatives of these functions appear; the function
$W$ goes to zero at infinity, together with
derivatives, the derivatives of $f_j$ are bounded.
Then, by means of commuting the operators $\Q$ with
the functions, we move all entries of $\Q$ in the
first and in the second factors in \eqref{3:Qform1} to
utmost right position, where they vanish since $\Q
u=0$ for $u\in \Hc_0$. The only remaining term in the
form $\wF[u]$ will be
\begin{equation}\label{3:Qform2}
\wF[u]=(W_1  u,  u),
\end{equation}
where $W_1$ is a function (a combination of $W,
f_j,\bb$ and their derivatives) tending to zero at
infinity. Now take  $\ve>0$ and represent $W_1$ as
$W_1=W_{1,\ve}+W'_{1,\ve}$ so that $|W'_{1,\ve}|<\ve$
and $W_{1,\ve}$ has compact support. For
$(W'_{1,\ve}u,u)$, we have the estimate by $\ve
||u||^2$, so the corresponding operator has norm not
greater than $\ve$. For $(W_{1,\ve}u,u)$, we take some
$\Rb$ such that the support of the function
 $W_{1,\ve}$ lies inside the circle $C_\Rb$ with radius $\Rb$
 centered in the origin. For each
 $r\in(\Rb,2\Rb)$ we write the Cauchy  representation
 for an analytical function $h(z)$:
 \begin{equation}\label{3:IntegralRepr1}
    h(z)=(2\pi i)^{-1}
    \int_{C_r}h(\z){(z-\z)}^{-1}
    d{\z}.
\end{equation}
for some fixed function $\x(r)\in
C_0^\infty(\Rb,2\Rb)$, $\int \x(r)dr =1$, we multiply
\eqref{3:IntegralRepr1} by $\x(r)$ and integrate in
$r$ from $\Rb$ to $2\Rb$. This gives the integral
representation of $h(x),\; |x|<\Rb$, in the form
$h(x)=$ $\int_{\Rb<|y|<2\Rb} K(x,y)h(y) dy$, with
smooth bounded kernel $K(x,y)$. After applying $\LF$
in $x$ variable, we obtain the representation for $\LF
u=\LF(e^{-\Psi}h)$:
\begin{equation}
\nonumber
    \LF u(x)=\int_{\Rb<|y|<2\Rb}
    e^{-\Psi(x)}K^{\LF}(x,y)e^{\Psi(y)}u(y)dy=(\KF^{\LF}u)(x).
\end{equation}
The integral operator $|W_{1,\ve}|^{1/2}\KF^{\LF}$ has
a bounded kernel with compact support and therefore
 is compact in
$L_2$, and thus the quadratic form $\wF[u]$ can be
written as $\wF[u]=(\sgn W_{1,\ve}
|W_{1,\ve}|^{1/2}\KF^{\LF}u,$ $
|W_{1,\ve}|^{1/2}\KF^\LF u)$ and therefore it is
compact. Now we see that the quadratic form
\eqref{3:Qform} can be for any $\ve$ represented as
the sum of a compact form and a form with norm less
than $\ve$, and this proves the required
compactness.\end{proof}

Now we are able to justify our construction of the
spaces $\Gc_q$.

\begin{proposition}\label{3:SubspacesGood} The sets
$\Gc_q$ defined in \eqref{3:Subspaces} are closed
subspaces in $L_2$.\end{proposition}
\begin{proof} The fact that $\Gc_q\subset L_2$ follows directly
from Proposition~\ref{3:basicPropositionBP}. Next, the
relation \eqref{3:BP:equation} can be written as
\begin{equation}\label{3:Subspaces:1}
    (P_0 \QF^q \Qa^q u,u)=C_q(u,u)+(P_0\Zc_q[\bb]
    u,u);\; u=P_0 u\in\Hc_0.
\end{equation}
 In the second term in
\eqref{3:Subspaces:1}, by Lemma~\ref{3:LemComp}, the
operator $P_0\Zc_q[\bb]$ is compact in $\Hc_0$, and
therefore we can understand \eqref{3:Subspaces:1} as
showing that the operator $C_q^{-1} P_0 \Q^q $ is a
left parametrix for $\Qa^q:\Hc_0\to L_2$. This implies
that the range of $\Qa^q$ is closed.\end{proof}

The null space of $\Qa$ and therefore of $\Qa^q$ is
zero. Consider the operator $\Qa^q$ as acting from
$\Gc_0=\Hc_0$ to $\Gc_q$. This is a bounded invertible
operator, therefore the inverse, that we denote by
$\Qa^{-q}$, is a bounded operator from $\Gc_q$ to
$\Gc_0$.  It is a compact perturbation of $\Q^q$.

\section{Approximate spectral projections}\label{Projections4}

In this section we prove that the subspaces $\Gc_q$
are very good approximations to the  spectral
subspaces $\Hc_q$ of the operator $\PF_-$, and to the
spectral subspaces of $\PF_-(V)$. Closeness of
subspaces will be measured by closeness of orthogonal
projections onto them. Recall that the projection onto
$\Hc_q$ is denoted by $P_q$. Let $Q_q$ be the
 projection onto $\Gc_q$.
\begin{theorem}\label{4:ThmCloseness}The projections
$P_q$ and $Q_q$ are \emph{close}: for any $N$, and any
collection of the operators $T_j,\; j=1,\dots,N$, each
 of $T_j$ being $\Q$ or $\Qa$,
the operator $\Tc(P_q-Q_q)\Tc$, is compact,
$\Tc=T_1T_2\dots T_N$.\end{theorem}

In justifying the theorem, we need two technical
lemmas, both concerning the  properties of products
 of many copies of the
resolvents of $\PF_+$ and $\PF_-$, the creation and
 annihilation operators $\Qa$ and $\Q$,   functions
$h_j$ with all derivatives bounded, and, possibly, the
 spectral projection $P_q$.
In such product, we assign order $1$ to $\Qa$ and
$\Q$, order $-2$ to the resolvent, order $0$ to
functions and projections. The order of the product is
defined as the sum of orders of factors.

\begin{lemma}\label{4:LemmaComp1}Let  $\AF$ be
 the product
  of
creation, annihilation operators, resolvents, and
functions $h_j$, have negative order, and let at least
one of the functions $h_j$ belong to $S_\n, \ \n<0$.
Then $\AF$ is compact.\end{lemma}
\begin{lemma}\label{4:LemmaComp2} Let  $\AF$ be
the product of creation, annihilation operators,
resolvents, functions and the projection $P_q$. Then
$\AF$ is bounded. If, moreover,
  at least one of $h_j$ belongs to $S_\n, \ \n<0,$ then $\AF$ is compact.\end{lemma}

  The proof of the earlier  versions of these lemmas, with the $S_\n$ - condition
  replaced
   by the compactness of
  the support of one of the function is given in
  \cite{RozTa}. In the present formulation, the proof
  goes exactly in the same way, only using the new
  version of Lemma~\ref{2:LemmaBound} and Lemma \ref{2:LemmaBound.1}, fit for the
  relaxed conditions for the functions.

The proof of the theorem  goes in the following way.
We  construct an intermediate operator $\SF_q$ with
range in $\Gc_q$ and prove that $\SF_q$ is close both
to $P_q$ and $Q_q$. For $q>0$ we define the operator
$\SF_q$ as
\begin{equation}\label{4:ApproxProj}
    \SF_q=C_q^{-1}\Qa^q P_0 \Q^q,\; C_q=q!(2\Bn)^q.
\end{equation}
So, our expression for the approximate spectral
projection is just a natural modification of the exact
formula \eqref{2:Projections} for the unperturbed
operator. Equivalently, the operator $\SF_q$ can be
described by the formula
\begin{equation}
\nonumber
 \SF_q=C_q^{-1} \GF \GF^*,\qquad \GF=\Qa^qP_0;
\end{equation}
in
Proposition~\ref{3:basicPropositionBP} this operator is shown to be bounded.

The \emph{proof} of Theorem~\ref{4:ThmCloseness} will
consist of two parts, showing that $\SF_q$ is close to
$P_q$ and showing that it is close to $Q_q.$ The
second part is proved exactly like in \cite{RozTa},
since it is based upon Lemma \ref{4:LemmaComp2} only.
As for the first part, we need a more detailed
information of the difference $\SF_q-P_q$. This
information is given in the following statement.

\begin{proposition}\label{4:Prop:SclosetoP} The
operator $\SF_q$ is close to the projection $P_q$.
Moreover, the difference $\SF_q-P_q$ has the the form
\begin{equation}\label{4:SclosetoPbest}
\SF_q-P_q = C_q(\PF_--\L_q)P_q+\Zb_q,
\end{equation}
where $\Zb_q$ is such an operator that $\langle
x\rangle^{-\b+\d} \Tc\Zb_q \Tc''$ is bounded for any
$\Tc,\Tc'$ being finite products of creation and
annihilation operators.
\end{proposition}
So, the improvement, compared with Proposition 4.4 in
\cite{RozTa}, consists, first, in the separation of
the leading term in the difference, the operator
$(\PF_--\L_q)P_q$, and, secondly, in the establishing
the improved smallness of the remainder term $\Zb_q$

\begin{proof} Combinatorically, the reasoning goes in the same way as
 in \cite{RozTa}, with a natural replacement of the auxiliary results
  requiring the compactness of the support of the perturbation
   and singling out the leading term. We remind the
   main structure of the reasoning, omitting the details.
   The case $q=1$ is considered first; this case
   contains all typical features. The higher Landau
   levels, $q>1$, can be taken care of in the same way
   as in \cite{RozTa}, by the induction on $q$.

 Recall that  $R_\pm(\z)$ denotes the
resolvent of the operator $\PF_\pm$. The projection
$P_0$ can be expressed via Riesz integral
\begin{equation}
\nonumber P_0=(2\pi i)^{-1}\int_{\G_0}\Rm(\z)d\z,
\end{equation}
where $\G_0$ is the closed curve defined, together
with curves $\G_q$, in Sect.3. We are going to
transform the expression for the resolvent $\Rm(\z)$
using the commutation relations \eqref{2:ComRel.1},
\eqref{2:ComRel.2}, \eqref{2:ComRel.3}. After this,
the crucial observation is that the integral of
$R_\pm^k,\, k\ge2$ along $\G_q$ vanishes. This enables
us to dispose of terms that  before integration were
not weak enough. We start by writing
\begin{gather}\label{4:TransRes2}
\Rm(\z)=(\PF_+ -2\Bn -2 \bb-\z)^{-1}=(\PF_+ -2\Bn
-\z)^{-1}-\\ \notag (\PF_--\z)^{-1}(2\bb) (\PF_+ -2\Bn
-\z)^{-1}=\Rp(2\Bn+\z)-Z(\z)
\end{gather}
We multiply  \eqref{4:TransRes2} by
 $\Qa$ from the left and by $\Q$ from the right, as \eqref{4:ApproxProj}
  requires. For the first term  we use that
\begin{gather}\label{4:TransRes2.1}\Qa \Rp(2\Bn+\z)\Q=
\Qa(\Q\Qa -2\Bn -\z)^{-1}\Q=\\ \notag \Qa
\Q(\Qa\Q-2\Bn-\z)^{-1}=\PF_-\Rm(2\Bn+\z).\end{gather}

Integration gives
\begin{equation*}
\nonumber
    \int_{\G_0}\Qa(\PF_- -2\Bn
    -\z)^{-1}\Q d\z=\PF_- \int_{\G_1}(\PF_-
    -\z)^{-1} d\z =(2\pi i)\PF_-P_1.
\end{equation*}

So, we have
\begin{equation}\label{4:TransResA1}
    \SF_1-P_1=\L_q^{-1}(\PF-\L_1)P_1 +\Zb_1,
\end{equation}
where $\Zb_1$ is the contour integral of the term
$Z(\z)$ in \eqref{4:TransRes2}.

 As explained in Section 2, the nonzero
eigenvalues of $\PF_-P_1$ may converge only to $\L_1$,
and thus $(\PF_--\L_1)P_1$ is compact by
Proposition~\ref{2:PropComp}. Moreover, this
compactness is preserved after the multiplication by
any product of creation and annihilation operators.

 Now we
consider the remainder term,  $Z(\z)$ in
\eqref{4:TransRes2}. We are going  to
 show that for any operators of the type $\Tc, \Tc'$
  the integral of $\langle
x\rangle^{-\b+\d}\Tc Z(\z)\Tc'$
 along  $\G_0$ is bounded. Due to
  the arbitrariness of $\Tc, \Tc'$ this ,
  of course, implies the compactness
 required by the theorem.

 Let $N$ be some fixed, sufficiently large integer.
 We apply the resolvent
formula \eqref{4:TransRes2} $2N$ times to the first
factor in $Z(\z)$. This operation will produce terms
of order $-4$, $-6$, \dots,
 containing  factors $\Rm(\z)$
and $\bb$, and the remainder term of order $-2N$
containing these factors with, additionally, one
factor $\Rp(2\Bn+\z)$. This last  remainder term, for
$N$ sufficiently large,
 is compact by
Lemma~\ref{4:LemmaComp1}, even after the
multiplication by $\Tc$, $\Tc'$.  We will study its
spectral properties below.

 The leading terms in $Z(\z)$, having orders $-4$,
 $-6$,\dots,
will be transformed by repeatedly commuting $\bb$ and
$\Rm(\z)$ and then the resulting commutants again with
$\Rm(\z)$ and so on. Under commuting $\Rm(\z)$ with a
function, with $\Q$, or with $\Qa$, this factor
$\Rm(\z)$ moves to the left or to the right, and one
more product in the sum composing $\Tc Z(\z) \Tc'$
arises, of the order lower by $1$. In this commuting
procedure we aim for collecting the factors $\Rm(\z)$
together. As soon as we obtain a term with all
$\Rm(\z)$ standing together, we leave it alone and do
not transform any more. After sufficiently many
commutations, we arrive at a collection of terms  of
negative order, smaller than $-N$,  in $\Tc
H(\z)\Tc'$. all of which are compact by
Lemma~\ref{4:LemmaComp1}. The terms of  order $-N$ or
higher
  will have the form $G_1\Rm(\z)^kG_2$ with $k>1$ and
   some operators $G_1,G_2$. These terms vanish
 after integration along $\G_0$.

 Let us consider the structure of all remaining operators of
 order lower than $-N$ again. They will contain $\bb$
 and its derivatives as factors. The only term where
 only one factor $\bb$ is present, as it follows from
 \eqref{4:TransRes2}, will be the one of the form
 $2\bb (\PF_--\z)^{-2}$, obtained by commuting the first
 and second factors in $(\PF_--\z)(2\bb)(\PF_--\z)$.
 The integral of this term along the contour vanishes.
 All the remaining terms will contain at least two
 entries of $\bb$ or at least one derivative of $\bb$.
 By Proposition \ref{2:commutatorEstimate}, used with
 proper $l$, this means that each such term is bounded
 as acting into the  weighted space with weight
 $\langle
x\rangle^{-\b+\d}$. This boundedness property is
preserved after the contour integration.

The detailed combinatorics in the above reasoning as
well as the inductive procedure enabling us to pass
from $q=1$ to an arbitrary $q$, is explained in
\cite{RozTa}.
\end{proof}
Thus the operator $\SF_q$ is close to the projection
$P_q$. Together with the closeness of $\SF_q$ to
$Q_q$, this reasoning proves Theorem
\ref{4:ThmCloseness}.

Now we add a perturbation by a smooth  electric
potential $V(x)\in S_\b$. Similar to \cite{RozTa}
under such perturbation the spectral subspaces
'almost' do not change.
  Note, first of all, that the perturbation of $\PF_-$ by $V$ is
  relatively compact, therefore, again, the spectrum of the
   operator $\PF_-(V)=\PF_-+V$ between Landau levels
   is discrete. We  can change  the  contours
   $\G_q$ a little, so that they do not pass through the
   eigenvalues of $\PF_-(V)$.
We denote by $\Hc^V_q$ the spectral subspaces of
$\PF_-(V)$
 corresponding to the
spectrum inside $\G_q$, by $P^V_q$ the corresponding
spectral projections, and by $\Rm^V(\z)$ the resolvent
of $\PF_-(V)$.

\begin{proposition}\label{4:PropElPert} The projections
$P^V_q$ and $Q_q$ are close in the sense used  in
Theorem~\ref{4:ThmCloseness}. Moreover, $\langle
x\rangle^{-\b+\d}\Tc (P^V_q-Q_q)\Tc'$ is bounded for
any finite products $\Tc,\Tc'$ of creation and
annihilation operators.
\end{proposition}
\begin{proof} We will prove that the projection
$P^V_q$ is close to $P_q$, then the result will follow
from Theorem~\ref{4:ThmCloseness}.

We use the representation of projections $P^V_q$ and
$Q_q$ by means of resolvents and subtract:

\begin{gather}\label{4:PropElPert.2}
  \langle x\rangle^{-\b+\d}  P^V_q-P_q =-(2\pi i)^{-1}\langle x\rangle^{-\b+\d}\int_{\G_q}\Rm(\z)V
    \Rm^V(\z)d\z=\\ \notag
(2\pi i)^{-1}\sum_{k=1}^{2N-1}\int_{\G_q}
    \langle x\rangle^{-\b+\d}(-\Rm(\z)V)^k\Rm(\z)d\z \\ \notag+(2\pi
    i)^{-1}\int_{\G_q}\langle x\rangle^{-\b+\d}(-\Rm(\z)V)^N\Rm^V(\z)(-V\Rm(\z))^N
    d\z.
\end{gather}
The last term in \eqref{4:PropElPert.2} is bounded and
stays bounded after the multiplication by the creation
and annihilation operators, as soon as $N$ is large
enough, by Lemma~\ref{4:LemmaComp1}.  With the leading
terms in \eqref{4:PropElPert.2}, we can perform the
same procedure as when proving
Proposition~\ref{4:Prop:SclosetoP}. We  commute the
resolvent with $V$ and with the terms arising by
commutation and so on, aiming to collect the
resolvents together all the time. We arrive at a
number of terms of sufficiently negative order, thus
bounded before the  integration, and terms with all
entries of the resolvent collected together, and thus
vanishing after the integration. The boundedness
property with weight follows again, as in the proof of
Proposition \ref{4:Prop:SclosetoP}, from the fact that
after the commutation, all surviving terms will
contain at least two entries of $V$ or an entry of a
derivative of $V$.
\end{proof}

\section{Spectrum of Toeplitz-type operators}\label{Toeplitz}

We move on to the study of the splitting of Landau
levels of our operators. Similar to \cite{RaiWar},
\cite{MelRoz}, \cite{RozTa}, \cite{IwaTam1},
 the properties of the eigenvalues of
the perturbed operators are determined by the
properties of the spectrum of certain Toeplitz-type
operators.

Usually, by Toeplitz operator one understands an
operator of the form $\TF(W)=PWP$ where $P$ is the
orthogonal projection onto some subspace $\Gc$ in
$L_2$ and $W$ is the operator of multiplication by
some function. Alternatively, if we consider the
Toeplitz operator as acting in $\Gc$ it can be written
as $\TF(W)=PW$. Usually, the subspace $\Gc$ consists of  functions, related with analytical ones, for example, Hardy or Bergman spaces.

In \cite{RaiWar}, \cite{MelRoz} such operators were
considered, with $\Gc$ being one of Landau subspaces
and the electric potential $V$ acting as $W$. We will
study, as in \cite{RozTa}, this latter kind of
Toeplitz type operators, with the space $\Hc_0$ of
zero modes of the perturbed Pauli operator acting as
$\Gc$ and some differential operator acting as $W$.
The result below differs essentially from the one in
\cite{RozTa}.

 Let  $V, \bb$,  be  real
functions in $S_\b$, $\b<-2$. We consider  the
Toeplitz-type operator in $\Hc_0$:
\begin{equation}\label{5:Above.1}
    \TF_0 u= \TF_0(V) u = P_0 \Q^q (\PF_--\L_q+ V) \Qa^q u,\; u\in \Hc_0.
\end{equation}
This operator, by \eqref{2:ComRel.1}, corresponds to
the quadratic form
\begin{gather}\nonumber
    \tF^{V}[u]=((\PF_--\L_q+ V) \Qa^q u, \Qa^q u)=
     (\PF_+-\L_{q+1}) \|\Qa^{q+1}u\|^2\\ \label{5:Above.2} -
     \L_{q+1}\|\Qa^{q}u\|^2  +((V-2\bb) \Qa^q u, \Qa^q u)\;,u\in \Hc_0.
\end{gather}

We are going to study the spectrum of $\TF_0(V)$ . We
denote by $\l_n^{\pm}=\l_n^{\pm}(\TF_0(V))$ positive,
 resp., negative eigenvalues of the operator
 $\TF_0(V)$. The distribution functions $n_\pm(\l),\, \l>0,$
 are
 defined as $n_\pm(\l)=n_\pm(\l;\TF_0(V))=$ $\#\{n:\;\pm
 \l_n^{\pm}>\l\}$.
 The s-numbers $s_n$
of the operator $\TF_0(V)$ are just the absolute values
 of $\l_n^{\pm}$ ordered non-increasingly, and their
 distribution function equals $n(\l,\TF_0(V))=n_+(\l,\TF_0(V))+n_-(\l,\TF_0(V))$.

We denote by $\Vb=\Vb_q[V,\bb]$ the \emph{effective
weight}
\begin{equation}\label{5:Vb(x)}
\Vb[V,\bb]= C_q(V+2q\bb), \ C_q=q!(2\Bn)^q,
\end{equation}
and  set for any function $W$ $E_\pm(\l,W)=(2\pi)^{-1}\Bn\meas\{x\in\R^2:
\pm W(x)>\l\}$.
\begin{proposition}\label{5:EstimateAboveL}
For the eigenvalue distribution function $n_\pm(\l;\TF_0(V))$
of the operator $\TF_0(V)$ the following  estimate
holds:
\begin{equation}\label{5:EstimateGeneral}
n_\pm(\l;\TF_0(V))\le C\l^{\frac2\b},\ \l\to0
\end{equation}
Moreover, suppose that   the function
$E_\pm(\l,\Vb)$ (for one or for both signs)
satisfies the estimate
\begin{equation}\label{5:E>}
    E_\pm(\l,\Vb)\ge C' \l^{\frac2\b}, \l\to0
\end{equation}
and the regularity condition
\begin{equation}\label{5:E-regular}
    \lim_{\e\to
    0}\limsup_{\l\to 0}\frac{E_\pm(\l(1-\e),\Vb)}{E_\pm(\l,\Vb)}=1.
\end{equation}
 Then for the distribution
function $n_\pm(\l,\TF_0(V))$ of the eigenvalues of $\TF_0(V)$
the asymptotic formula holds
\begin{equation}\label{5:ToeplitzAsymp}
    n_\pm(\l,\TF_0(V))\sim E_\pm(\l,\Vb), \ \l\to 0.
\end{equation}
\end{proposition}

The conditions of the form  \eqref{5:E>} and
\eqref{5:E-regular} are traditional in the study of
asymptotic properties of operators with (possibly)
non-power behavior of eigenvalues, see, e.g.,
\cite{Raikov1}, \cite{IwaTam1}, \cite{IwaTam2}. The
first of them indicates that the 'rate of decay' $\b$
is chosen sharply, so that various remainder terms
are, in fact, weaker than the leading one. The second
condition means that the function $E_\pm(\l,\Vb)$
grows sufficiently regularly. It enables the use of
various  kinds of perturbation techniques.

We note also that for $q=0$ and $V=0$, the effective
weight $\Vb$ vanishes. This property corresponds to the
fact that $\PF_-$ has $\Hc_0$ as its  space of zero
modes.

\begin{proof}
By Proposition \ref{3:basicPropositionBP},  the  expression
\eqref{5:Above.2}  can be transformed to
$$\tF^{V}[u]=(\Vb u,u)+(\vb u,u),\ u\in \Hc_0,$$
where $\vb\in S_{\b-\d}$ and $\Vb=\Vb_q[V,\bb]\in S_\b$. By Lemma
2.5 in \cite{IwaTam1},  for the operator defined by
this quadratic form, the asymptotic relation
$$n_{\pm}(\l)\sim (2\pi)^{-1}\Bn \meas\{x\in\R^2:\Vb(x)+\vb(x)\ge\l\}, \l\to 0,$$
holds, as soon as the measure on the right-hand side
grows not slower than $\l^{\frac2\b}$ and is regular,
i.e., if the conditions  \eqref{5:E>},
\eqref{5:E-regular} are satisfied. This reasoning
takes care of the  second part of the Proposition. The
first part follows  from Proposition 2.3 in
\cite{IwaTam1}, without qualified lower estimate. The
function $\vb$, due to its decay, does not contribute
to the main order of the eigenvalue estimates and
asymptotics.
\end{proof}

Now we can  establish  the spectral  estimates and
asymptotics for a similar Toeplitz operator on the
spectral subspace corresponding to the cluster around
the Landau level $\L_q$ for an arbitrary $q$.

\begin{proposition}\label{5:Toeplitz_q}
Let $b, V\in S_\b, \ \b<-2$. Consider the
Toeplitz-type operator
\begin{equation}\label{5:Toeplitz_q1}
    \TF_q(V)=P_q(\PF_--\L_q+V)P_q.
\end{equation}
Then for the distribution function of s-numbers of the
eigenvalues of $\TF_q(V)$ the estimate holds
\begin{equation}\label{5:Toeplitz_q2}
    n(\l,\TF_q(V))=O(\l^{\frac2\b}).
\end{equation}
If, moreover, the effective weight $\Vb$ defined in
\eqref{5:Vb(x)} satisfies the conditions \eqref{5:E>},
\eqref{5:E-regular} then the asymptotic formula
\begin{equation}\label{5:Toeplitz.qq}
    n_\pm(\l,\TF_q(V))\sim E_{\pm}(\l,V+2q\bb)
\end{equation}
 holds.
\end{proposition}

\begin{proof}
 By Proposition \ref{4:Prop:SclosetoP}, the spectral
projection $P_q$ can be approximated by $\SF_q=C_q
\Qa^q P_0 \Q^q$. We use \eqref{4:SclosetoPbest} to
obtain
\begin{gather}\nonumber
\TF_q(V)=\SF_q(\PF_--\L_q+V)\SF_q
+(P_q(\PF_--\L_q)P_q+\Zb_q)(\PF_--\L_q+V)\SF_q\\\label{5:Toeplitzq.1}
+\SF_q(\PF_--\L_q+V)(P_q(\PF_--\L_q)P_q+\Zb_q) +\\
\nonumber(P_q(\PF_--\L_q)P_q+\Zb_q)(\PF_--\L_q+V)(P_q(\PF_--\L_q)P_q+\Zb_q).
\end{gather}
We re-arrange \eqref{5:Toeplitzq.1} to separate the
leading terms:
\begin{gather*}
    \TF_q(V)=\SF_q(\PF_--\L_q+V)\SF_q+P_q(\PF_--\L_q+V)P_q(\PF_--\L_q+V)P_q\SF_q\\ +\SF_q
    P_q(\PF_--\L_q+V)P_q(\PF_--\L_q+V)P_q+\\
P_q(\PF_--\L_q+V)P_q(\PF_--\L_q+V)P_q(\PF_--\L_q+V)P_q
+\Yb.
\end{gather*}
So, we have
\begin{equation}\label{5:Toeplitzq.3}
\TF_q(V)=\SF_q(\PF_--\L_q+V)\SF_q
+\TF_q(V)^2\SF_q+\SF_q\TF_q(V)^2+\TF_q(V)^3 +\Yb.
\end{equation}
Let us consider the structure of the remainder
operator $\Yb$. On the one hand, it contains terms
having the factor $\Zb_q$. By Proposition
\ref{4:Prop:SclosetoP}, the operator $\langle
x\rangle^{-\b+\d}(\PF_+)^N\Zb_q$ is bounded for $N$
that can be chosen arbitrarily large. Thus, by
Proposition \ref{4a:low order}, the s-numbers of
$\Zb_q$ satisfy the estimate
 $n(\l,\Zb_q)=o(\l^{\frac2\b}), \ \l\to 0$. Therefore
 all terms in $\Yb$ containing $\Zb_q$ satisfy the
 same kind of s-numbers estimate.

 The terms in $\Yb$ not containing $\Zb_q$, contain as
 factors the function
 $V$, the projection $P_q$ and some compact operators
 that remain compact after the multiplication by any power of
 $\PF_+$. For the operator $V(\PF_+)^{-N}$, by
 Proposition \ref{4a:low order}, the estimate $n(\l,
 V(\PF_+)^{-N})=O(\l^{\frac2\b})$ holds. After the
 multiplication by a compact operator, the $O$ symbol
 can be replaced by $o$ in this formula. As a result,
 we obtain
 \begin{equation}\label{5:Toeplitzq.4}
    n(\l, \Yb)=o(\l^{\frac2\b}), \l\to 0.
\end{equation}

Now consider the second, third, and fourth terms on
the right-hand side in \eqref{5:Toeplitzq.3}. They
contain the square of the operator $\TF_q(V)$
multiplied by some bounded operators. Therefore the
rate of decay of $s$-numbers of these terms is faster
than the rate of decay of the ones of $\TF_q(V)$. So,
the estimate \eqref{5:Toeplitz_q2} or the asymptotic
formula \eqref{5:Toeplitz.qq} will be proved as soon
as we establish these formulas for the operator
$\Db=\SF_q(\PF_--\L_q+V)\SF_q$. The quadratic form of
this operator is $(\Db u,u)=((\PF_--\L_q+V)\SF_q u,
\SF_q u)$. Recalling the definition of $\SF_q$ in
\eqref{4:ApproxProj}, we have
\begin{equation*}
    (\Db u,u)=C_q(\Qa^q P_0 \Q^q u, (\PF_--\L_q+V)\Qa^q P_0 \Q^q
    u).
\end{equation*}
We set here $v=P_0\Q^q u, \ v\in \Hc_0$, to get
\begin{equation}\label{5:Toeplitzq.5}
(\Db u,u)=C_q^{-1}(\Qa^q v, (\PF_--\L_q+V)\Qa^q v), \
v\in \Hc_0.
\end{equation}

The spectral estimate and, under the conditions
\eqref{5:E>}, \eqref{5:E-regular}, the  asymptotics
for the operator defined by the latter quadratic form
is given by Proposition \ref{5:EstimateAboveL}, where the factor $C_q^{-1}$ is responsible for the replacement of $\Vb$ by $C_q^{-1}\Vb=V+2q\bb$. This
result carries over to the quadratic form $(\Db u, u)$
using the fact that $u=C_q(1+K)\Qa^q v$ for some
compact operator $K$, which is explained in
\cite{RozTa}.
\end{proof}

The operators of the above type will be used in the
next Section in order to find the leading term in the
eigenvalue asymptotics in clusters. Another type of
Toeplitz operators will be needed in order to perform
a block digitalization of the Pauli operator.

\begin{proposition}\label{5:Prop.counterdiag}Let $W$ be a function in $S_\b, \b<-2$.
Consider the operator $\TF'(W)=(1-P_q)W P_q$. Then for
the distribution function of singular numbers of
$\TF'(W)$
\begin{equation}\label{5:counterdiag1}
    n(\l, \TF'(W))=o(\l^{\frac{2}{\b}}), \ \l\to 0.
\end{equation}
\end{proposition}
\begin{proof}  The quadratic form of the operator
$\TF'(W)^*\TF'(W)$ equals
\begin{gather}\label{5:counterdiag2}
    (\TF'(W)u,\TF'(W)u)=((1-P_q)Wu,(1-P_q)Wu)\\ \nonumber =([W,P_q]P_qu, [W,P_q]P_qu), u\in \Hc_q.
\end{gather}
 Consider the case $q=0$ first. By Proposition \ref{4a:low order} the operator $[W,P_0]$ can be
represented as
\begin{equation}\label{5:countergiagA}
[W,P_0]= L\langle x\rangle^{\b-\d}
\end{equation}
for some bounded operator $L$.    By Proposition
\ref{5:EstimateAboveL}, the operator $[W,P_0]$ has
singular numbers with the required decay rate.

For an arbitrary $q>0$ by Proposition
\ref{2:commutatorEstimate},
\begin{equation}\label{5:countergiagB}
[W,P_q]= L\langle x\rangle^{\b-\d};
\end{equation}
We use the approximation of the projection $P_q$ found
in Section \ref{Projections4}:
\begin{equation}\label{5:counterdiagC}
    [W,P_q]P_q= L\langle x\rangle^{\b-\d} \SF_q+ L\langle x\rangle^{\b-\d}
    (P_q-\SF_q),
\end{equation}
where $\SF_q$ is defined in \eqref{4:ApproxProj}. The
spectral estimate for the first,  leading term in
\eqref{5:counterdiagC} involves the projection $P_0$
and the task of estimating its singular numbers
reduces to the already considered case $q=0$, see
\eqref{5:Toeplitzq.5}. By Proposition
\ref{4:Prop:SclosetoP}, the second term in
\eqref{5:counterdiagC} can be written as $L\langle
x\rangle^{\b-\d}\PF_+^{-N}K$ with as large $N$ as
needed and a compact operator $K$. The required
spectral estimate for this operator follows now from
the second part of Proposition~\ref{4a:low order}.
\end{proof}

\section{Perturbed eigenvalues}
Now we are able to establish our main result about
eigenvalue asymptotics and estimates for the perturbed
Schr\"odinger and Pauli operators. For a self-adjoint
 operator $L$, we denote by $N(\l,\m)=N(\l,\m;L)$ the
 number of eigenvalues of $L$ in the interval
 $(\l,\m)$.

\begin{theorem}\label{6:MainTheo}
Let $V,\bb\in S_\b, \b<-2$. Fix an integer $q\ge0$ and
let $\l_\pm$ be some fixed real numbers,
 $\l_{\pm}\gtrless \L_q$, $|\l_\pm-\L_q|<\Bn$, such that
they are not the points of the spectrum of $\PF_-+V$.

For  $\PF_-(V)$ the following estimates hold for
$\l\to 0+$
\begin{equation}\label{6:MainTheoEst}
N(\L_q+\l,\l_+;\PF_-(V))=O(\l^{\frac2\b}),\
N(\l_-,\L_q-\l; \PF_-(V))=O(\l^{\frac2\b}).
\end{equation}
If the effective weight $\Vb_q[V,\bb]=C_q(V+2q\bb)$ or, what is equivalent, $V+2q\bb$
 satisfies the conditions
\eqref{5:E>}, \eqref{5:E-regular} with the sign "$+$"
then, asymptotically,
\begin{equation}\label{6:MainTheoAs+}
    N(\L_q+\l,\l_+;\PF_-(V))\sim E_+(\l,V+2q\bb),\ \l\to 0,\
    \l>0.
\end{equation}
If the effective weight satisfies the
conditions \eqref{5:E>}, \eqref{5:E-regular} with the
sign "$-$" then, asymptotically,
\begin{equation}\label{6:MainTheoAs-}
    N(\l_-,\L_q-\l,;\PF_-(V))\sim E_-(\l,V+2q\bb),\ \l\to 0,\
    \l>0.
\end{equation}

Similar results hold for the  operator $\PF_+( V)$ and
for the Schr\"odinger operator $\HF(V)$  with the
following obvious modifications, corresponding to
\eqref{3:DifferentOperators}. For the Pauli operator
$\PF_+ (V)$ one should replace in the estimates of the
form \eqref{6:MainTheoEst} and in the asymptotic
relations of the form $\L_q$ by $\L_{q+1}$, $\l_\pm$
by $\l_\pm+2\Bn$ and $V$ by $V+2\bb$. For  the
Schr\"odinger  operator $\HF(V)$  one should replace
$\L_q$ by $\L_q+\Bn$, $\l_\pm$ by $\l_\pm+\Bn$ and $V$
by $V+\bb$.\end{theorem}
\begin{proof} We will use the following statement about
the eigenvalue distribution of perturbed operators. If
$L_0,L_1$ are two self-adjoint operators, moreover,
$L_1$ is compact, then
\begin{gather}\label{6:perturb}
N(\m_1,\m_2;L_0+L_1)\le N(\m_1-\t_1,
\m_2+\t_2;L_0)+n(\t_1;L_1)+n(\t_2,L_1)
\end{gather}
for any interval $(\m_1,\m_2)$ and any positive
numbers $\t_1,\t_2$. The proof of \eqref{6:perturb}
can be found, e.g., in \cite{Raikov1}, Lemma 5.4.

We show now the upper asymptotic estimate in
\eqref{6:MainTheoAs+}. For a fixed $q$, we take as
$L_0$ the operator
\begin{equation}\label{6:L0}L_0=P_q(\PF_-+V)P_q+(1-P_q)(\PF_-+V)(1-P_q),\end{equation}
and as $L_1$ the operator
$$L_1=(1-P_q)VP_q+P_qV(1-P_q),$$
so that $L_0+L_1=\PF_-+V$. We fix some $\e>0$, and
apply \eqref{6:perturb} for $\m_1=\L_q+\l, \l_+$,
$\t_1=\t_2=\e\l$. The spectrum of $L_0$ is the union
of the spectra of the summands in \eqref{6:L0}. The
asymptotics of the eigenvalues of $P_q(\PF_-+V)P_q$ is
given by the Proposition \ref{5:Toeplitz_q}. The
second term in \eqref{6:L0} contributes to the
spectrum near $\L_q$ only with finitely many  points.
Thus, $N(\L_q+\m_1-\t_1, \L_q+\m_2+\t_2;L_0)\sim
E_+(\l(1-\e); V+2qb)$. On the other hand, for the
spectrum of the operator $L_1$, by Proposition
\ref{5:Prop.counterdiag}, we have $n(\e\l,
L_1)=o((\e\l)^{\frac2\b})$. We substitute these
asymptotic estimates into \eqref{6:perturb}, divide by
$E_+(\l,V+2q\bb)$ and pass to $\limsup$ as $\l\to 0$. We
arrive at
$$\limsup_{\l\to0} \frac{N(\L_q+\l, \l_+ , \PF_-+V)}{E_+(\l(1-\e); V+2q\bb)}\le 1.$$
Due to the arbitrariness of $\e$, by our assumptions,
this implies the upper asymptotic estimate in
\eqref{6:MainTheoAs+}.

 All other upper estimates in
the Theorem are established analogously. The lower
asymptotic estimate in \eqref{6:MainTheoAs+},
\eqref{6:MainTheoAs-} is established in the same way,
just interchanging $L_0$ and $L_1$ in
\eqref{6:perturb}.
\end{proof}

 \end{document}